\documentclass[a4paper,11pt]{amsart}
\usepackage[english]{babel}
\usepackage{amsmath}
\usepackage{amssymb}
\usepackage{mathrsfs}
\usepackage{enumerate}
\usepackage{mathtools}
\usepackage{tikz,tkz-euclide}
\usepackage{pgfplots}
\usepackage{hyperref}
\usetikzlibrary{arrows}
\usepackage{makecell}

\usepackage{tcolorbox}
\usepackage[]{algorithm2e}

\newtheorem{thm}{Theorem}[section]
\newtheorem{Lemma}[thm]{Lemma}
\newtheorem{Proposition}[thm]{Proposition}

\newtheorem{Claim}[thm]{Claim}
\newtheorem*{thm*}{Theorem}

\theoremstyle{definition}
\newtheorem{Notation}[thm]{Notation}

\newtheorem{Definition}[thm]{Definition}
\newtheorem{Remark}[thm]{Remark}

\newtheorem{say}[thm]{}

\definecolor{wwwwww}{rgb}{0.4,0.4,0.4}

\newcommand{\Spec}{\operatorname{Spec}}

\renewcommand{\P}{\mathbb{P}}
\newcommand{\F}{\mathbb{F}}

\newcommand{\p}{\mathbb{P}}

\DeclareMathOperator{\Cox}{Cox}
\DeclareMathOperator{\Pic}{Pic}

\DeclareMathOperator{\Eff}{Eff}

\DeclareMathOperator{\Gal}{Gal}

\DeclareMathOperator{\Nef}{Nef}

\hypersetup{pdfpagemode=UseNone}
\hypersetup{pdfstartview=FitH}
		
\begin{document}

\title{On the birational geometry of conic bundles over the projective space}

\author[Alex Massarenti]{Alex Massarenti}
\address{\sc Alex Massarenti\\ Dipartimento di Matematica e Informatica, Universit\`a di Ferrara, Via Machiavelli 30, 44121 Ferrara, Italy}
\email{msslxa@unife.it}

\author[Massimiliano Mella]{Massimiliano Mella}
\address{\sc Massimiliano Mella\\ Dipartimento di Matematica e Informatica, Universit\`a di Ferrara, Via Machiavelli 30, 44121 Ferrara, Italy}
\email{mll@unife.it}

\date{\today}
\subjclass[2020]{Primary 14E08, 14M20, 14M22; Secondary 14J30, 14E30, 14J26.}
\keywords{Conic bundles, Calabi-Yau pairs, unirationality}

\begin{abstract}
Let $\pi:Z\rightarrow\mathbb{P}^{n-1}$ be a general minimal $n$-fold
conic bundle with a hypersurface $B_Z\subset\mathbb{P}^{n-1}$
of degree $d$ as discriminant. We prove that if $d\geq 4n+1$ then $-K_Z$ is not
pseudo-effective, and that if $d = 4n$ then none of the integral multiples of $-K_{Z}$ is effective.
Finally, we provide examples of smooth unirational $n$-fold conic bundles
$\pi:Z\rightarrow\mathbb{P}^{n-1}$ with discriminant of arbitrarily
high degree. 
\end{abstract}

\maketitle
\setcounter{tocdepth}{1}
\tableofcontents

\section{Introduction}
A projective variety $X$ is rationally connected if through any two
general points of $X$ there is an irreducible rational curve $C\subset
X$. We refer to \cite{Ar05} for a comprehensive survey on the
subject. All smooth Fano varieties, over an algebraically closed field,
are rationally connected \cite{Ca92}, \cite{KMM92}. While there are
finitely many families of smooth Fano varieties, rationally connected
varieties come in infinitely many families, and indeed there are
rationally connected varieties that are not birational to a Fano
variety, for example birationally rigid conic bundles over rational
surfaces \cite{Sa80}, \cite{Co95}, \cite[Theorem 4]{BCZ04}. 

A projective variety $X$ is of Fano type if it is normal and admits
a $\mathbb{Q}$-divisor $\Delta\subset X$ such that the pair
$(X,\Delta)$ is Kawamata log terminal and $-K_X+\Delta$ is
ample. These can be viewed as a log version of Fano varieties. Even
though varieties of Fano type come in infinitely many families up to
birational equivalence \cite{Ok09}, there are examples of rationally
connected varieties that are not birational to a variety of Fano type
\cite{Kr18}. This answers to a question raised in \cite{CG13}. The
counterexamples in \cite{Kr18} are constructed as double covers of
$\mathbb{P}^m\times\mathbb{P}^1$ branched along divisors of
sufficiently big and even bi-degree.      

Further weakening the positivity hypotheses on $-K_X$ we get to the
notion of numerical Calabi-Yau pair. Following \cite{Ko17} we define a
numerical Calabi-Yau pair as a pair $(X,\Delta)$ where $X$ is a normal
projective variety and $\Delta\subset X$ is a pseudo-effective
$\mathbb{Q}$-divisor such that $K_X +\Delta$ is a
$\mathbb{Q}$-Cartier numerically trivial divisor. The question of
whether or not any rationally connected variety is birational to the
underlying variety of a Calabi-Yau pair has been addressed for the
fist time by J. Koll\'{a}r in \cite{Ko17} taking into account the
specific class of rationally connected varieties formed by conic
bundles over $\p^2$. 

Nowadays, thanks to the minimal model program \cite{Mo88},
\cite{BCHM10}, we know that any uniruled projective variety is
birationally equivalent to a mildly singular projective variety $Z$
admitting a contraction $\pi:Z\rightarrow W$, of relative Picard
number one, to a lower dimensional normal projective variety $W$ such
that $-K_Z$ is $\pi$-ample. When $\dim(Z) = \dim(W)+1$ we say that
$\pi:Z\rightarrow W$ is a conic bundle. By \cite[Corollary 1.3]{GHS03}
a conic bundle over a rational variety is rationally connected. We say
that $\pi:Z\rightarrow W$ is extremal if its relative Picard rank is one, and that it is minimal if in addition it does not have rational sections. The fundamental invariant of a conic bundle is the
discriminant divisor $B_Z$, that is the closure of the locus of $W$ over
which the fibers of $\pi:Z\rightarrow W$ are reducible. 

In this paper we study the pseudo-effectiveness of the anti-canonical
divisor of minimal conic bundles over projective spaces. Recall that a
$\mathbb{Q}$-divisor on a normal projective variety $X$ is
pseudo-effective if its numerical class is a limit of effective $\mathbb{Q}$-divisors. We say that $-K_X$ is birationally effective if there exists a projective normal variety $Y$,
birational to $X$, and such that $-K_{Y}$ is pseudo-effective. Similarly, we say that $-K_X$ is birationally
effective if some integral multiple of $-K_{Y}$ is effective. 

In \cite{Ko17} J. Koll\'{a}r proved that if $\pi:Z\rightarrow\mathbb{P}^2$ is a general $3$-fold conic bundle with a curve $B_Z\subset\mathbb{P}^2$ of degree $d\geq 19$ as discriminant then $-K_Z$ is not
birationally pseudo-effective. In general, it is an open problem to
establish under which conditions a conic bundle is birational to the
underlying variety of a numerical Calabi-Yau pair \cite[Question
14.5.1]{Pr18}. In this direction our main results can be summarized as follows,  see Paragraph
\ref{gen} for a precise definition of general.

\begin{thm}\label{main1}
Let $\pi:Z\rightarrow\mathbb{P}^{n-1}$ be a general minimal $n$-fold
conic bundle with discriminant $B_Z\subset\mathbb{P}^{n-1}$ of
degree $d$, and $n\geq 3$. If $d\geq 4n+1$ then $-K_Z$ is not
pseudo-effective, and if $d = 4n$ then none of the integral multiples of $-K_{Z}$ is effective.  
\end{thm}

As shown in \cite[Example 20]{Ko17} and in Remark \ref{ur_hd} there are conic bundles $\pi:Z\rightarrow\mathbb{P}^{n-1}$, with discriminant of arbitrarily high degree, whose anti-canonical
divisor $-K_Z$ is effective. Hence, a genericity assumption in Theorem
\ref{main1} is needed.

Building on \cite[Example 20]{Ko17} we investigate the unirationality of a certain class of
conic bundles.  
Recall that an $n$-dimensional variety $X$ over a field $k$ is
rational if it is birational to $\mathbb{P}^n_{k}$, while $X$ is
unirational if there is a dominant rational map
$\mathbb{P}^m_{k}\dasharrow X$. The L\"uroth problem, asking whether
every unirational variety was rational, dates back to the second half
of the nineteenth century \cite{Lu75}. These two notions turned out to
be equivalent for curves and complex surfaces. Only in the $1970$s
examples of unirational but non-rational $3$-folds, over an
algebraically closed field of characteristic zero, were given by
M. Artin and D. Mumford \cite{AM72}, V. Iskovskih and I. Manin
\cite{IM71} and  C. Clemens and P. Griffiths \cite{CG72}. Note that
unirational varieties are rationally connected.  

Rationality of $3$-fold conic bundles has been extensively studied and a
precise conjectural rationality statement is at hand, we refer to
\cite{Pr18} for a comprehensive survey. The further notion of stable
rationality has been treated in \cite{HKT16} and \cite{AO18}. On the contrary
unirationality for $3$-fold conic bundles is still widely open and not
much is known. We refer to \cite{KM17} for the unirationality of a
special class of conic bundles.  Classically, conjectures on
unirationality of conic bundles take into account the degree of the
discriminant and, mimicking what is known about rationality, conic
bundles with general discriminant of large degree are expected to be
non-unirational \cite[Section 14.2]{Pr18}. In Section \ref{ur} we
prove the following result. 

\begin{thm}\label{main2}
For all $n\geq 3$ there are unirational, smooth and extremal $n$-fold conic bundles $\pi:Z\rightarrow\mathbb{P}^{n-1}$ with discriminant $B_Z\subset\mathbb{P}^{n-1}$ of arbitrarily high degree.
\end{thm}

The unirational conic bundles in Theorem \ref{main2} are explicitly constructed as hypersurfaces in toric varieties isomorphic to the projectivization of suitable splitting vector bundles over the projective space. Smooth extremal $3$-fold conic bundles $\pi:Z\rightarrow\mathbb{P}^{2}$ with discriminant of degree at least six are not rational, see for instance \cite[Theorem 9.1]{Pr18}. In particular, they do dot have rational sections. Therefore, by Theorem \ref{main2} we get that there are smooth, minimal and unirational $3$-fold conic bundles over the projective plane with discriminant of arbitrarily high degree. In Remark \ref{ur_hd} we observe that Theorem \ref{main2} can be extended to surface conic bundles when the base field is not algebraically closed.

\subsection*{Organization of the paper} 
Unless otherwise stated we work over a perfect field $k$ of characteristic different from two. The paper is organized as follows. In Section \ref{CB_H} we recall the main definitions and facts about conic bundles and Hirzebruch surfaces. In Section \ref{main_sec} we prove Theorem \ref{main1}. Finally, in Section \ref{ur} we prove Theorem \ref{main2}.

\subsection*{Acknowledgments} 
The authors thank Ciro Ciliberto, Antonio Laface and Stefan Schreieder for helpful discussions, and are members of the Gruppo Nazionale per le Strutture Algebriche, Geometriche e le loro Applicazioni of the Istituto Nazionale di Alta Matematica "F. Severi" (GNSAGA-INDAM). We thank the referee for pointing out a mistake in a statement, about a birational version of Theorem \ref{main1} for $3$-folds, that we removed.

\section{Conic bundles and Hirzebruch surfaces}\label{CB_H}
In this section we recall the definitions and results about Hirzebruch
surfaces and conic bundles which we will use in the following.

\begin{Definition} Let $\mathbb{F}_e=\mathbb{P}_{\mathbb{P}^1}(\mathcal{O}_{\P^1}\oplus\mathcal{O}_{\P^1}(-e))$ be a Hirzebruch surface, and $\pi:\mathbb{F}_e\rightarrow\mathbb{P}^1$ the projection with general fiber $F$. We will denote by
\begin{itemize}
\item[-] $C_0\subset \mathbb{F}_e$ the only section with negative self intersection if $e>0$, and a fixed section if $e=0$;
\item[-]  and by $F_p:=\pi^{-1}(\pi(p))$ the fiber of $\pi$ through the point $p\in \mathbb{F}_e$. 
\end{itemize}
\end{Definition}

\subsection{Elementary transformations of Hirzebruch surfaces}\label{el_2}
Take a point $p\in \mathbb{F}_e$, and let $\psi:X_p\rightarrow\mathbb{F}_e$ be the blow-up of $\mathbb{F}_e$ at $p$. The strict transform $\widetilde{F}_p$ of $F_p$ is a $(-1)$-curve. Hence, there is a birational morphism $\phi:X_p\rightarrow \F_r$ contracting $\widetilde{F}_p$. The birational map $\phi\circ\psi^{-1}:\mathbb{F}_e\dasharrow \mathbb{F}_r$ is the elementary transformation centered at $p\in \mathbb{F}_e$. Note that of $p\notin C_0$ then $r = e-1$, while $p\in C_0$ implies that $r = e+1$.

In accordance with \cite{Ko17} we introduce the following definitions.

\begin{Definition}\label{def:cb}
A conic bundle is a flat and projective morphism $\pi:Z\rightarrow W$, between normal and projective varieties, with fibers isomorphic to plane conics. 

A conic bundle $\pi:Z\rightarrow W$ is extremal if the relative Picard number of $\pi$ is one, and it is minimal if it is extremal and does not have rational sections.

The discriminant $B_Z$ is the subscheme of $W$ parametrizing singular  fibers of $\pi:Z\rightarrow W$. Set-theoretically this is the locus of points $w\in W$ such that $Z_w := \pi^{-1}(w)$ is singular. 
\end{Definition}

Note that if $Z$ is smooth then $\pi:Z\rightarrow W$ is minimal if and only if $\Pic(Z)\cong \pi^{*}\Pic(W)\oplus\mathbb{Z}[\omega_{Z/W}^{-1}]$. In this case, if $\mathcal{L}$ is a line bundle on $Z$ then $\mathcal{L}$ has even degree on a general fiber of $\pi$, and $\mathcal{L}$ has the same degree on the two components of a reducible fiber.  

\begin{say}\label{gen}
Now, we construct explicit examples of conic bundles and rigorously define the genericity condition in Theorem \ref{main1}.

Let $a_0,a_1,a_2 \in\mathbb{Z}_{\geq 0}$, with $a_0\geq a_1\geq a_2$, be non-negative integers, and consider the simplicial $\mathbb{Q}$-factorial toric variety $\mathcal{T}_{a_0,a_1,a_2}$ with Cox ring
$$\Cox(\mathcal{T}_{a_0,a_1,a_2})\cong k[x_0,\dots,x_{n-1},y_0,y_1,y_2]$$
$\mathbb{Z}^2$-grading given, with respect to a fixed basis $(H_1,H_2)$ of $\Pic(\mathcal{T}_{a_0,a_1,a_2})$, by the following matrix
$$
\left(\begin{array}{cccccc}
1 & \dots & 1 & -a_0 & -a_1 & -a_2\\ 
0 & \dots & 0 & 1 & 1 & 1
\end{array}\right)
$$
and irrelevant ideal $(x_0,\dots,x_{n-1})\cap (y_0,y_1,y_2)$.
Then
$$\mathcal{T}_{a_0,a_1,a_2}\cong \mathbb{P}(\mathcal{E}_{a_0,a_1,a_2})$$
with
$\mathcal{E}_{a_0,a_1,a_2}\cong\mathcal{O}_{\mathbb{P}^n}(a_0)\oplus
\mathcal{O}_{\mathbb{P}^n}(a_1)\oplus
\mathcal{O}_{\mathbb{P}^n}(a_2)$.
The presentation in Cox coordinates allows us to write global equations for conic bundles in $\mathcal{T}_{a_0,a_1,a_2}$ as follows:
\stepcounter{thm}
\begin{equation}\label{Cox_g}
X_{\mathcal{T}} := \{\sigma_{d_0}y_0^2+2\sigma_{d_1}y_0y_1+2\sigma_{d_2}y_0y_2+\sigma_{d_3}y_1^2+2\sigma_{d_4}y_1y_2+\sigma_{d_5}y_2^2 = 0\}\subset \mathcal{T}_{a_0,a_1,a_2}
\end{equation}
where $\sigma_{d_i}\in k[x_0,\dots,x_n]_{d_i}$ is a homogeneous polynomial of degree $d_i$, and 
$$d_0-2a_0 = d_1-a_0-a_1 = d_2-a_0-a_2 = d_3-2a_1 = d_4-a_1-a_2 = d_5-2a_2.$$
Note that the discriminant of $X_{\mathcal{T}}$ has degree $d_0+d_3+d_5$. Let $\pi_{X_{\mathcal{T}}}:X_{\mathcal{T}}\rightarrow\mathbb{P}^{n-1}$ be the restriction to $X_{\mathcal{T}}$ of the projection $\mathcal{T}_{a_0,a_1,a_2} = \mathbb{P}(\mathcal{E}_{a_0,a_1,a_2 })\rightarrow\mathbb{P}^{n-1}$.

Now, let $\pi:Z\rightarrow\mathbb{P}^{n-1}$ be a conic bundles with
discriminant of degree $d$, $L\subset\mathbb{P}^{n-1}$ a general line, $S_L = \pi^{-1}(L)$, $\mu:S_L\rightarrow\mathbb{F}_e$ a blow-down morphism, and
$p_1,\dots,p_d\in\mathbb{F}_e$ the blown-up points. We say that
$\pi:Z\rightarrow\mathbb{P}^{n-1}$ is general if the points
$p_1,\dots,p_d\in\mathbb{F}_e$ are in general position. We will now construct conic bundles for which this generality condition holds. Reasoning backwards begin with general points $p_1,\dots,p_d\in\mathbb{F}_e$ and denote by $S$ their blow-up. Then $S$ is a surface conic bundle and hence it can be embedded in a projective bundle 
$$\mathbb{P}(\mathcal{O}_{\mathbb{P}^1}(a_0)\oplus
\mathcal{O}_{\mathbb{P}^1}(a_1)\oplus
\mathcal{O}_{\mathbb{P}^1}(a_2))\rightarrow\mathbb{P}^1$$
and written in the form (\ref{Cox_g}):
$$S = \{\sigma_{d_0}y_0^2+2\sigma_{d_1}y_0y_1+2\sigma_{d_2}y_0y_2+\sigma_{d_3}y_1^2+2\sigma_{d_4}y_1y_2+\sigma_{d_5}y_2^2 = 0\}$$
with $\sigma_{d_i}\in k[x_0,x_1]_{d_i}$. Now, consider polynomials $s_{d_i}\in k[x_0,\dots,x_{n-1}]_{d_i}$ such that 
$$s_{d_i}(x_0,x_1,0,\dots,0) = \sigma_{d_i}$$ 
and the conic bundle 
$$
X_{\mathcal{T}} := \{s_{d_0}y_0^2+2s_{d_1}y_0y_1+2s_{d_2}y_0y_2+s_{d_3}y_1^2+2s_{d_4}y_1y_2+s_{d_5}y_2^2 = 0\}\subset \mathcal{T}_{a_0,a_1,a_2}.
$$ 
Then $X_{\mathcal{T}|L_0} = S$, where $L_0 = \{x_2 = \dots = x_{n-1} = 0\}$, determines the $d$ general points $p_1,\dots,p_d\in\mathbb{F}_e$ we started with. Therefore, if $L\subset\mathbb{P}^{n-1}$ is a general line then $\mu:S_L\rightarrow\mathbb{F}_e$ yields $d$ general point on $\mathbb{F}_e$. We would like to clarify that this argument works only in the scenario of this particular example, and not for any conic bundle over $\mathbb{P}^{n-1}$.
\end{say} 

\begin{Remark}\label{con_fib}
Note that when $n\geq 7$ the locus $\Sigma=\{\sigma_{d_0} = \sigma_{d_1} =\dots = \sigma_{d_5} = 0\}\subset\mathbb{P}^{n-1}$ is not empty, and the fiber of $X_{\mathcal{T}}$ over each point of $\Sigma$ is the whole of $\mathbb{P}^2$. So $\pi_{X_{\mathcal{T}}}:X_{\mathcal{T}}\rightarrow\mathbb{P}^{n-1}$ is not flat, and hence it is not a conic bundle in the sense of Definition \ref{def:cb}. However, by abusing terminology, in Section \ref{ur} we will refer to such non flat fibrations as conic bundles as well.  

As long as the vector $(d_0-2a_0,2)$ lies in the cone generated by $(1,0)$ and $(-a_2,1)$ the divisor $(d_0-2a_0)H_1+2H_2$ is nef on $\mathcal{T}_{a_0,a_1,a_2}$. Hence, by \cite[Theorem 6.3.12]{CLS11} $X_{\mathcal{T}}\subset \mathcal{T}_{a_0,a_1,a_2}$ is cut out by a section of a globally generated divisor. So, when the $\sigma_{d_i}$ are general, by Bertini's theorem \cite[Corollary 10.9]{Ha77} $X_{\mathcal{T}}$ is smooth. Finally, when $n\geq 3$ and $(d_0-2a_0)H_1+2H_2$ is ample the Grothendieck-Lefschetz theorem yields that the Picard rank of $X_{\mathcal{T}}$ is two. 
\end{Remark}

\subsection{Elementary transformations of $3$-fold conic bundles}\label{el_3}
Let $\pi:Z\rightarrow W$ be a smooth $3$-fold conic bundle over a smooth surface. Take a reduced curve $\Gamma\subset W$ such that $\pi$ is smooth over the generic points of $\Gamma$, and let $s:\Gamma\dasharrow Z_{\Gamma}:= \pi^{-1}(\Gamma)\subset Z$ be a rational section of $\pi_{|\Gamma}$. By blowing-up $\overline{s(\Gamma)}$ and then contracting, possibly after a sequence of flops, the strict transform of $Z_{\Gamma}$ we get another smooth $3$-fold conic bundle $\pi':Z'\rightarrow W$ with the same discriminant curve of $\pi:Z\rightarrow W$ \cite[Section 26]{Ko17}.

\begin{Definition}\label{bir_pse}
Let $X$ be a projective normal variety. A $\mathbb{Q}$-divisor on $X$ is effective if it has an integral multiple that is effective. A $\mathbb{Q}$-divisor on $X$ is pseudo-effective if its numerical class in $N_{\dim(X)-1}(X)\otimes_{\mathbb{Z}}\mathbb{Q}$ is a limit of effective $\mathbb{Q}$-divisors.

We say that $-K_X$ is birationally pseudo-effective if there exists a projective normal variety $Y$, birational to $X$, and such that $-K_{Y}$ is pseudo-effective. Similarly, we say that $-K_X$ is birationally effective if some integral multiple of $-K_{Y}$ is effective. 
\end{Definition}

Finally, we recall the following unirationality criterion due to
F. Enriques \cite[Proposition 10.1.1]{IP99} that will be used in Section \ref{ur}.
\begin{Proposition}\label{Enr}
Let $\pi:Z\rightarrow W$ be a conic bundle. Then $Z$ is unirational if and only if there exists a unirational subvariety $D\subset Z$ such that $\pi_{|D}:D\rightarrow W$ is dominant. 
\end{Proposition}

\section{Numerical Calabi-Yau pairs and  conic bundles over the projective space}\label{main_sec}

Let $S\rightarrow\mathbb{P}^1$ be a minimal conic bundle surface over a non algebraically closed field $k$ with discriminant $B_S$ of degree $\delta_S$, and general  fiber $F$.

\begin{Lemma}\label{K_1}
Assume that $\delta_S > 8 + 4r$, where $r\in\mathbb{N}$, and $-K_S+rF$ is pseudo-effective . Then all pseudo-effective $\mathbb{Q}$-divisors on $S$ are effective. In particular $m(-K_S+rF)$ is effective for some integer $m\geq 1$.
\end{Lemma}
\begin{proof}
Let $D_t$ be a sequence of effective $\mathbb{Q}$-divisors converging to $-K_S+rF$. Since $K_S^2 = 8-\delta_S$ and $K_S\cdot F = -2$, $\delta_S > 8 + 4r$ yields 
$$(-K_S+rF)^2 = 8-\delta_S + 4r < 0.$$
Therefore, $D_t^2 < 0$ for some $t$, and hence there exists an irreducible curve $C$, defined over $k$, and contained in the support of $D_t$ such that $C^2 <0$.  

Now, to conclude it is enough to note that since $F^2 = 0$ and $C^2 < 0$ the cone of curves of $S$ is the closed cone generated by $F$ and $C$, so all pseudo-effective $\mathbb{Q}$-divisors are effective.
\end{proof}

Let $k$ be a field, $\overline{k}$ its algebraic closure and $\Gal(\overline{k}/k)$ the corresponding Galois group. For a projective variety $X$ over $k$ we will denote by
$$\overline{X} := X\times_{\Spec(k)}\Spec(\overline{k})$$
its algebraic closure. 

\begin{Remark}\label{cbc}
If $D$ is a Cartier divisor on $X$ and $\overline{D}$ is the divisor on $\overline{X}$ obtained by base change then
$$H^0(\overline{X},\mathcal{O}_{\overline{X}}(\overline{D}))\cong H^0(X,\mathcal{O}_{X}(D))\otimes_{k}\overline{k}.$$ 
This is indeed a consequence of Cohomology and Base Change \cite[Theorem 12.11]{Ha77}.
\end{Remark}

The following result is a consequence of Kleiman’s criterion for ampleness over an arbitrary field \cite[Theorem 3.9]{Ke03}. However, for the convenience of the reader, we present an alternative proof.  

\begin{Lemma}\label{nef_K}
Let $X$ be a projective variety over a field $k$, and $D$ a nef divisor on $X$. Then $\overline{D}\subset\overline{X}$ is nef as well, and $D$ is a limit of ample divisors on $X$.
\end{Lemma}
\begin{proof}
Since $D$ is nef we have that $D\cdot C\geq 0$ for all curves $C\subset X$ defined over $k$. Let $C_0\subset \overline{X}$ be a curve, and $C_1,\dots,C_r$ its conjugates with respect to the action of $\Gal(\overline{k}/k)$. Then $C = C_0\cup C_1\cup\dots C_r\subset\overline{X}$ is defined over $k$, and hence $D\cdot C = D\cdot C_0 + D\cdot C_1 +\dots + D\cdot C_r\geq 0$. 

Furthermore, since $D$ is defined over $k$ it is $\Gal(\overline{k}/k)$-invariant, and since the $C_i$ are conjugates we conclude that $D\cdot C = (r+1)(D\cdot C_0)$. Hence, $D\cdot C_0\geq 0$ and $D$ is nef over $\overline{k}$. If $A$ is an ample divisor on $X$ and $\overline{A}$ is the corresponding ample divisor on $\overline{X}$ the sequence of ample divisors $t\overline{A}+D$ converges to $D$ on $\overline{X}$, and hence the divisors $tA+D$ yield a sequence of ample divisors on $X$ converging to $D$.  
\end{proof}

\begin{Remark}\label{ps_K}
Let $X$ be a projective variety over a field $k$, and $D$ a nef divisor on $X$. Then $D$ is pseudo-effective. Indeed by Lemma \ref{nef_K} the nef divisor $D$ is a limit of a sequence of ample divisors $D_t$. Hence, the $\mathbb{Q}$-divisors $D_t$ are effective, and $D$ is then pseudo-effective.
\end{Remark}

Let $\pi:Z\rightarrow\mathbb{P}^{n-1}$ be a minimal $n$-fold conic bundle, and consider the following incidence variety
\[
  \begin{tikzpicture}[xscale=1.5,yscale=-1.5]
    \node (A0_1) at (1, 0) {$W = \{(z,L)\: | \: z\in \pi^{-1}(L)\}\subseteq Z\times \mathbb{G}(1,n-1)$};
    \node (A1_0) at (0, 1) {$Z$};
    \node (A1_2) at (2, 1) {$\mathbb{G}(1,n-1)$};
    \path (A0_1) edge [->]node [auto] {$\scriptstyle{\psi}$} (A1_2);
    \path (A0_1) edge [->]node [auto,swap] {$\scriptstyle{\phi}$} (A1_0);
  \end{tikzpicture}
  \]
The generic fiber $S_Z:=W_{\eta}$ of $\psi: W\rightarrow\mathbb{G}(1,n-1)$ is a minimal surface conic bundle over the function field $k(t_1,\dots,t_{2n-4})$. We will denote by 
$$\overline{S}_Z := S_Z \times_{\Spec(k(t_1,\dots,t_{2n-4}))}\Spec(\overline{k(t_1,\dots,t_{2n-4})})$$ 
the algebraic closure of $S_Z$. Finally, set $S_L := \pi^{-1}(L)$ where $L\subset\mathbb{P}^{n-1}$ is a line. 

\begin{Lemma}\label{can_re}
Set $-K_{Z|S_Z} := \phi^{*}(-K_{Z})_{|S_Z}$. Then $-K_{Z|S_Z} = -K_{S_Z} + (n-2)F$.
\end{Lemma}
\begin{proof}
Let $L\subset\mathbb{P}^{n-1}$ be a general line. When $n = 3$ we have $\mathcal{N}_{S_L/Z} = S_{L|S_L} = F$. We will prove, by induction on $n$, that $\det(\mathcal{N}_{S_L|Z}) = (n-2)F$ for all $n\geq 3$.

Take a general hyperplane $H\subset\mathbb{P}^{n-1}$ and set $Z_{H}:= \pi^{-1}(H)$. Then $\pi_{|Z_H}:Z_H\rightarrow H$ is a conic bundle on dimension $n-1$ over $\mathbb{P}^{n-2}$. Take $L\subset H$ a general line, and consider the exact sequence 
$$0\rightarrow \mathcal{N}_{S_L/Z_H}\rightarrow \mathcal{N}_{S_L/Z}\rightarrow \mathcal{N}_{Z_H/Z|S_L}\rightarrow 0.$$
We have $\mathcal{N}_{Z_H/Z|S_L} = F$, and by induction $\det(\mathcal{N}_{S_L/Z_H}) = (n-3)F$. Hence, $\det(\mathcal{N}_{S_L|Z}) = (n-2)F$. 

Finally, from the exact sequence
$$0\rightarrow \mathcal{N}_{S_L/Z}\rightarrow \mathcal{T}_{Z|S_L}\rightarrow \mathcal{T}_{S_L}\rightarrow 0$$
we get that $-K_{Z|S_L} = -K_{S_L} + \det(\mathcal{N}_{S_L/Z}) = -K_{S_L} + (n-2)F$. 
\end{proof}

Our next aim is to prove the emptiness of certain linear systems
on the ruled surfaces $\F_e$. 

\begin{Notation}
We denote by $\mathcal{L}_{m,d}(\mathbb{F}_e)$ the linear system of curves in $|-mK_{\mathbb{F}_e}+m(n-2)F|$ having multiplicity $m$ at $d$ general points of $\mathbb{F}_e$. 
\end{Notation}
 
The first step in this direction is the following result.

\begin{Lemma}\label{lyn_Fe}
Assume that $n\geq 2$, and $e\leq n$. If $d = 4n$ then $\mathcal{L}_{m,d}(\mathbb{F}_e)$ has at most one section for all $m\geq 1$. Furthermore, if $d\geq 4n+1$ then  
$\mathcal{L}_{m,d}(\mathbb{F}_e)$ is empty for all $m\geq 1$.
\end{Lemma}
\begin{proof}
For $e\leq n$ the linear system $|-K_{\mathbb{F}_e}+(n-2)F|$ contains
a smooth and irreducible curve $\Gamma_e$. For $e = 0$ this is a curve
of bidegree $(2,n)$ on $\mathbb{F}_0\cong \mathbb{P}^1\times
\mathbb{P}^1$. For $e = 1$ the curve $\Gamma_1$ is the strict
transform of an irreducible curve of degree $n+1$ in $\mathbb{P}^2$
with a point of multiplicity $n-1$. For $e = 2$ we may take as
$\Gamma_2$ the strict transform of $\Gamma_1$ via an elementary
transformation centered at a point of intersection of
$\Gamma_1$ with $C_0\subset\mathbb{F}_1$ as described in (\ref{el_2}). In general for $e
= k\leq n$ we take as $\Gamma_k$ the strict transform of
$\Gamma_{1}$ via a sequence of elementary transformations centered at
$k-1$ points in $\Gamma_1\cap C_0\subset \mathbb{F}_1$.

A straightforward computation gives $\Gamma_{e}^2 = 4n$ and
$K_{\mathbb{F}e}\cdot \Gamma_{e} = -2n-4$, then the adjunction formula
yields 
$$2g(\Gamma_e)-2 = \Gamma_e\cdot(\Gamma_e+K_{\mathbb{F}e}) = 2n-4$$
and hence $g(\Gamma_e) = n-1\geq 1$.

Now, fix $d = 4n$ general points $p_1,\dots,p_{4n}\in\Gamma_e$, and
let $X_{4n}$ be the blow-up of $\mathbb{F}_{e}$ at these points with
exceptional divisors $E_i$. Denote by $\widetilde{\Gamma}_e$ the
strict transform of $\Gamma_e$. Since $\widetilde{\Gamma}_e\sim
-K_{X_{4n}}+(n-2)F$, twisting the structure exact sequence of $\widetilde{\Gamma}_e$
$$0\rightarrow \mathcal{O}_{X_{4n}}(-\widetilde{\Gamma}_e)\rightarrow\mathcal{O}_{X_{4n}}\rightarrow\mathcal{O}_{\widetilde{\Gamma}_e}\rightarrow 0$$ 
 by
$\mathcal{O}_{X_{4n}}((m+1)(-K_{X_{4n}}+(n-2)F))$ we get the following
exact sequence 
\begin{footnotesize}
\begin{equation}\label{f_es}
0\rightarrow \mathcal{O}_{X_{4n}}(m(-K_{X_{4n}}+(n-2)F))\rightarrow\mathcal{O}_{X_{4n}}((m+1)(-K_{X_{4n}}+(n-2)F))\rightarrow\mathcal{O}_{\widetilde{\Gamma}_e}((m+1)(-K_{X_{4n}}+(n-2)F))\rightarrow 0.
\end{equation}
\end{footnotesize}
The first piece of the long exact sequence in cohomology induced by (\ref{f_es}) reads as follows
$$
\begin{array}{ll}
0\rightarrow & H^0(X_{4n},\mathcal{O}_{X_{4n}}(m(-K_{X_{4n}}+(n-2)F))) \rightarrow H^0(X_{4n},\mathcal{O}_{X_{4n}}((m+1)(-K_{X_{4n}}+(n-2)F)))  \\ 
 & \rightarrow H^0(\widetilde{\Gamma}_e,\mathcal{O}_{\widetilde{\Gamma}_e}((m+1)(-K_{X_{4n}}+(n-2)F)))
\end{array} 
$$
Note that since there does not exist a curve in
$|-K_{\mathbb{F}_e}+(n-2)F|$ through $4n$ general points of
$\mathbb{F}_e$, and we took $p_1,\dots,p_{4n}$ general on $\Gamma_e$
we have $h^0(X_{4n},\mathcal{O}_{X_{4n}}(-K_{X_{4n}}+(n-2)F)) =
1$. Now, assume that
$h^0(\widetilde{\Gamma}_e,\mathcal{O}_{\widetilde{\Gamma}_e}((m+1)(-K_{X_{4n}}+(n-2)F)))
= 0$ for all $m\geq 1$. Then  
$$H^0(X_{4n},\mathcal{O}_{X_{4n}}(m(-K_{X_{4n}}+(n-2)F))) \cong H^0(X_{4n},\mathcal{O}_{X_{4n}}((m+1)(-K_{X_{4n}}+(n-2)F)))$$
for all $m\geq 1$, and by induction on $m$ we conclude that
$h^0(X_{4n},\mathcal{O}_{X_{4n}}(m(-K_{X_{4n}}+(n-2)F))) = 1$ for all
$m\geq 1$.  

The divisor $\mathcal{O}_{\widetilde{\Gamma}_e}(-mK_{X_{4n}}+m(n-2)F)$
has degree zero, so
$$h^0(\widetilde{\Gamma}_e,\mathcal{O}_{\widetilde{\Gamma}_e}(-mK_{X_{4n}}+m(n-2)F))\in\{0,1\}$$
depending on whether
$\mathcal{O}_{\widetilde{\Gamma}_e}(-mK_{X_{4n}}+m(n-2)F)$ is trivial
or not, and 
$\mathcal{O}_{\widetilde{\Gamma}_e}(-mK_{X_{4n}}+m(n-2)F)$ is trivial
only if  $\mathcal{O}_{\widetilde{\Gamma}_e}(-K_{X_{4n}}+(n-2)F)$ is
torsion.  

Let $X_{4n-1}$ be the blow-down of the exceptional divisor
$E_{4n}$. Denote by $\Gamma_e^{'}$ the push-forward of
$\widetilde{\Gamma}_e$. Then
$\mathcal{O}_{\Gamma_e^{'}}(-K_{X_{4n-1}}+(n-2)F)$ is a divisor
$D^{'}$ on $\Gamma_e^{'}$ of degree one. Moreover, since
$\Gamma_e^{'}$ has positive genus, for a general choice of $p\in
\Gamma_e^{'}$ the divisor $D^{'} - p$ is not a torsion divisor. Hence,
$$h^0(\widetilde{\Gamma}_e,\mathcal{O}_{\widetilde{\Gamma}_e}(-mK_{X_{4n}}+m(n-2)F))
= 0$$
and
$$h^0(X_{4n},\mathcal{O}_{X_{4n}}(-mK_{X_{4n}}+m(n-2)F)) = 1$$
for all $m\geq 1$.  

Summing-up we proved that for $4n$ general points on a general section
of $|-K_{\mathbb{F}_e}+(n-2)F|$ the linear system of curves in
$|-mK_{\mathbb{F}_e}+m(n-2)F|$ having multiplicity $m$ at those $4n$
points has just one section for all $m\geq 1$. By semi-continuity we
conclude that the linear system of curves in
$|-mK_{\mathbb{F}_e}+m(n-2)F|$ having multiplicity $m$ at $4n$ general
points of $\mathbb{F}_e$ has at most one section for all $m\geq 1$.  

Finally, since by the previous part of the proof
$\mathcal{L}_{m,4n}(\mathbb{F}_e)$ has at most one section, after
imposing one more general base point we get an empty linear
system. Hence, if $d\geq 4n+1$ and the base points are general
$\mathcal{L}_{m,d}(\mathbb{F}_e)$ is empty for all $m\geq 1$.   
\end{proof}

\begin{thm}\label{th:4nempty}
If $n\geq 3$ and $e\leq n$ then the linear system
$\mathcal{L}_{m,d}(\mathbb{F}_{e})$ is empty for all $m\geq 1$ and
$d\geq 4n$.  
\end{thm}
\begin{proof}
Thanks to Lemma~\ref{lyn_Fe} we have only to deal with the case  $d =
4n$. For this we use the degeneration method introduced in
\cite{CM98}, \cite{CM11} with a little variation that we now
explain. Consider the family $X$ obtained by blowing-up the family
$\mathbb{F}_e \times \mathbb{A}^1\rightarrow\mathbb{A}^1$ at a general
point, say $p$, lying on the central fiber. The new central fiber $X_0$ has then
two components: $Y$ isomorphic to $\mathbb{F}_e$ blown-up at a general
point $p\in\mathbb{F}_e$, and $Z$ isomorphic to
$\mathbb{P}^2$. Furthermore, $Y$ and $Z$ intersects along a smooth
rational curve $R$ which is a $(-1)$-curve in $Y$ and a line in $Z$.  

Choose $4n-4$ general points on $Y$ and four general points on
$Z$. We denote by $\mathcal{A}$ the strict transform of the linear
system of curves in $|m(-K_{\mathbb{F}_e}+(n-2)F)|$ having
multiplicity $2m$ at $p$ and $m$ at the images in $\mathbb{F}_e$ of
the chosen $4n-4$ general points on $Y$. Furthermore, let
$\mathcal{B}$ be the linear system of curves of degree $2m$ with
multiplicity $m$ at the four chosen general points on
$Z\cong\mathbb{P}^2$.  

Arguing as in the first part of the proof of Lemma \ref{lyn_Fe} we see
that $|-K_{\mathbb{F}_e}+(n-2)F|$ contains an irreducible curve
$\Gamma_e$ 
passing through the $4n-4$ fixed points and whose singular locus consists of a single node at $p$.

The condition $n\geq 3$ yields
$$g(\Gamma_e) = n-2\geq 1.$$
By construction $\Gamma_e^2 = 4n$, hence, if we
blow-up the node $p$ and the $4n-4$ fixed points, and denote by
$\widetilde{\Gamma}_e$ the strict transform of $\Gamma_e$ we have
$$\widetilde{\Gamma}_e^2 = 0.$$
This means that
$\mathcal{O}_{\widetilde{\Gamma}_e}(\widetilde{\Gamma}_e)$ is a line
bundle of degree zero but it is not torsion since the blown-up points
are general. Hence, for all $m\geq 1$ the linear system $\mathcal{A}$
consists of a single element, namely the curve $\Gamma_e$ taken with
multiplicity $m$.  

We continue as in the proof of \cite[Proposition
5.5.9]{CHMR13}. The restriction of the linear systems $\mathcal{A}$ and $\mathcal{B}$ to $R$ has
degree $2x$. Furthermore, $\mathcal{A}$ has dimension zero and
$\mathcal{B}$ has dimension $m$. Since $2m \geq m+1$ by transversality
of the restricted systems \cite[Section 3]{CM98} we get that the limit
linear system consists of the kernel systems $\mathcal{A}^{'}$ whose
elements are strict transforms of curves in
$|m(-K_{\mathbb{F}_e}+(n-2)F)|$ having multiplicity $2m+1$ at $p$ and
$m$ at the images in $\mathbb{F}_e$ of the $4n-4$ fixed points, and
$\mathcal{B}^{'}$ of curves of degree $2m-1$ with multiplicity $m$ at
the four chosen general points on $Z\cong\mathbb{P}^2$. 

Now, since $\mathcal{A}$ consists of a single curve which is the
strict transform of a curve having multiplicity exactly $2m$ at $p$ we
conclude that $\mathcal{A}^{'}$ is empty, Furthermore, since the
strict transform of $\mathcal{B}^{'}$, in the blow-up of $Z$ at the
four fixed points, intersects negatively the strict transform of a
general conic through the four points we get that $\mathcal{B}^{'}$ is
empty as well.  

Finally by semicontinuity, we conclude that the original linear system of curves in
$|m(-K_{\mathbb{F}_e}+(n-2)F)|$ having multiplicity $m$ at $4n$
general points of $\mathbb{F}_e$, with $n\geq 3$ and $e\leq n$, is
empty for all $m\geq 1$.    
\end{proof}

We conclude the part on surfaces with the following result.

\begin{Proposition}\label{d12}
Let $S\rightarrow\mathbb{P}^1$ be a minimal conic bundle with
$\delta_S = 4n$. Then $-K_S + (n-2)F$ is
pseudo-effective. Furthermore, if $n\geq 3$ and $S\rightarrow\mathbb{P}^1$ is general then $|-mK_S +m(n-2)F|$ is empty for all $m\geq 1$.
\end{Proposition}
\begin{proof} By construction  $(-K_S + (n-2)F)\cdot F = 2$ and
$$(-K_S + (n-2)F)^2 = 8-4n+4(n-2) = 0.$$
The conic bundle $S$ is minimal, then its cone of curves has only two
rays. Assume that there exists a curve $C\subset S$ such that $(-K_S + (n-2)F)\cdot C < 0$. Then the ray spanned by $-K_S + (n-2)F$ lies in between the ray spanned by $F$ and that spanned by $C$. So $-K_S + (n-2)F$ is $\mathbb{Q}$-effective and since $(-K_S + (n-2)F)^2 = 0$ it spans the other extremal ray of the effective cone. A contradiction, since $C$ would then lie outside of the effective cone. Therefore, $-K_S + (n-2)F$ is nef and hence pseudo-effective by Remark~\ref{ps_K}.

Let $\overline{S}$ be the algebraic closure of $S$, and take a
blow-down morphism $\overline{S}\rightarrow\mathbb{F}_e$, with
exceptional divisors $E_i$, such that the blown-up points do not lie
on the negative section $C_0$. By Lemma \ref{nef_K} the divisor
$-K_{\overline{S}}+(n-2)F$ is nef on $\overline{S}$.
On the other hand, since $-K_{\mathbb{F}_e}\sim 2C_0+(e+2)F$ we have   
$$-K_{\overline{S}}+(n-2)F \sim 2\widetilde{C}_0 + (e+n)F -\sum_{i=1}^{4n}E_i$$
and $(-K_{\overline{S}}+(n-2)F)\cdot \widetilde{C}_0 = -e+n$,
where $\widetilde{C}_0$ is the strict transform of $C_0$. Therefore, since $-K_{\overline{S}}+(n-2)F$ is nef we get that
$$e\leq n.$$  
The sections of $-mK_{\overline{S}}+m(n-2)F$ are in bijection with the 
curves in $|-mK_{\mathbb{F}_e}+m(n-2)F|$ having multiplicity $m$ at
$4n$ general points. So, Theorem~\ref{th:4nempty} yields that
$h^0(\overline{S},-mK_{\overline{S}}+m(n-2)F)=0$ for all $m\geq 1$,
and by Remark \ref{cbc} we conclude that  
$$h^0(S,-mK_{S}+m(n-2)F) =
h^0(\overline{S},-mK_{\overline{S}}+m(n-2)F)=0$$  
for all $m\geq 1$.
\end{proof}

We are ready to prove the main result on conic bundles that are not birational to the underlying variety of a numerical Calabi-Yau pair. 

\begin{thm}\label{noCYn}
Let $\pi:Z\rightarrow\mathbb{P}^{n-1}$ be a general minimal $n$-fold conic
bundle with discriminant $B_Z\subset\mathbb{P}^{n-1}$ of
degree $d$. If $n\geq 2$ and $d\geq 4n+1$ the divisor $-K_Z$ is not
pseudo-effective. If $n\geq 3$ and $d=4n$ then none of the integral multiples of $-K_{Z}$ is effective. 
\end{thm}
\begin{proof} Assume first that $d\geq 4n+1$. By Lemma \ref{can_re} we have $-K_{Z|S_Z} = -K_{S_Z} + (n-2)F$. Assume by contradiction that $-K_{Z}$ is pseudo-effective. Since the morphism $\phi:W\rightarrow Z$ in Lemma \ref{can_re} is dominant and has connected fibers $\phi^{*}(-K_{Z})$ is pseudo-effective, and then \cite[Theorem 6.2]{Pet12} yields that $-K_{Z|S_Z} := \phi^{*}(-K_{Z})_{|S_Z}$ is also pseudo-effective. Then $-K_{S_Z}+(n-2)F$ is pseudo-effective and Lemma \ref{K_1} yields that
$b(-K_{S_Z}+(n-2)F)$ is effective for some $b\gg 0$. Let
$m\in\mathbb{Z}$ be the smallest positive integer such that
$m(-K_{S_Z}+(n-2)F)$ is effective.  

By Remark \ref{cbc} a section $D$ of $m(-K_{S_Z}+(n-2)F)$ induces a section $\overline{D}$ of $m(-K_{\overline{S}_Z}+(n-2)F)$. 
Blowing-down a component of each of the $d$ reducible fibers of the
conic bundle $\overline{S}_Z\rightarrow\mathbb{P}^1$, over
$\overline{k(t_1,\dots,t_{2n-4})}$, we get a morphism
$\mu:\overline{S}_Z\rightarrow\mathbb{F}_{e}$, and we may assume that
none of the blown-up points lies on the negative curve
$C_0\subset\mathbb{F}_{e}$. Then the linear system on $\mathbb{F}_{e}$
of the curves of class $-mK_{\mathbb{F}_e}+m(n-2)F$ having
multiplicity $m$ at $d$ general points has a section  
$$C\sim -mK_{\mathbb{F}_e}+m(n-2)F \sim 2mC_0 + m(e+n)F.$$ 

\begin{Claim}\label{cl1}
The surface $\overline{S}_Z$ admits a blow-down onto an Hirzebruch surface $\mathbb{F}_e$ with $e\leq n$.
\end{Claim}
\begin{proof}
  Assume by
contradiction that for all blow-down maps $\overline{S}_Z\rightarrow\F_e$ we have $e>n\geq 2$.
Then
$$C_0\cdot (2mC_0 + m(e+n)F) = m(n-e) < 0$$ 
and  $C$ contains $C_0$. Denote by $C^{'}$ the residual curve,
and let $\Gamma,\Gamma^{'}\subset\overline{S}_Z$ be the strict
transforms of $C_0$ and $C^{'}$ respectively.   

Recall that $\overline{D} = \Gamma \cup\Gamma^{'}$ comes from a
divisor $D\subset S_Z$ defined over $k(t_1,\dots,t_{2n-4})$. Since
$S_Z$ is minimal $D$ can not split over $k(t_1,\dots,t_{2n-4})$, and
hence $\Gamma$ has a conjugate $\widetilde{\Gamma}$ which is a
component of $\Gamma^{'}$. Hence, $\overline{D} = \Gamma \cup
\widetilde{\Gamma}\cup D^{'}$. Since, $\Gamma$ and
$\widetilde{\Gamma}$ are conjugate we have that $\widetilde{\Gamma}$
is also a section. So, we may write  
$$\widetilde{\Gamma} \sim \Gamma + a F -\sum_{i=1}^d E_i$$
with $a\geq 0$. Therefore, 
$$D^{'}\sim \overline{D} - \Gamma - \widetilde{\Gamma} = (2m-2)\Gamma +(m(e+n)-a)F -(m-1)\sum_{i=1}^dE_i.$$
If $a = e+n + c$, for some $c\geq 0$, then
$$D' + cF \sim (2m-2)\Gamma + (m-1)(e+a)F - (m-1)\sum_{i=1}^dE_i\sim (m-1)(-K_{\F_e}+(n-s)F)$$
would be effective contradicting the minimality of $m$.
This shows that
\stepcounter{thm}
\begin{equation}
  \label{eq:aen}
  a < e+n.
\end{equation}
Since $\Gamma$ and $\widetilde{\Gamma}$ are conjugate we have 
$$\widetilde{\Gamma}^2 = -e+2a-d = -e = \Gamma^2.$$
This forces  $d = 2a$. By blowing-down $a$ components of the reducible
fibers among the ones intersecting $\Gamma$, and $a$ components of the
reducible fibers among the ones intersecting $\widetilde{\Gamma}$ we
get a morphism $\overline{S}_Z\rightarrow\mathbb{F}_{e^{'}}$ with
$e^{'}\geq n+1$. The images $A,A^{'}\subset \mathbb{F}_{e^{'}}$ of
$\Gamma,\widetilde{\Gamma}$ are sections with self-intersection
$$A^2 =
(A^{'})^2 = -e+a.$$
The sections $A$ and $A'$ are distinct and $e>n\geq 2$ therefore
$$-e+a >0.$$
This combined with (\ref{eq:aen}) yields
$$1\leq -e+a
< n.$$
Consider a composition of $-e+a$ elementary transformations (\ref{el_2}) $\mathbb{F}_{e^{'}}\dasharrow\mathbb{F}_h$ centered at $-e+a$ general
points of $A$, and let $B\subset\mathbb{F}_h$ be the strict transform
of $A$.

By assumption $e'\geq n+1$, and hence $h> 1$. On the other hand, by construction
$B\subset\mathbb{F}_h$ is a section of self-intersection $B^2 =
0$. This contradiction proves the claim.
\end{proof}
By Claim \ref{cl1} we may assume that there exists a blow-down
$\overline{S}_Z\rightarrow\mathbb{F}_e$ with $e\leq n$, and to
conclude it is enough to apply Lemma~\ref{lyn_Fe}.

Finally, in the case $d = 4n$ it is enough to note that by the proof of Proposition \ref{d12} the surface $\overline{S}_Z$ admits a blow-down onto an Hirzebruch surface $\mathbb{F}_e$ with $e\leq n$, and to apply Theorem \ref{th:4nempty}.
\end{proof}

\begin{Remark}\label{n=1}
When $n = 2$ and $d = 4n = 8$ the second statement in Theorem \ref{noCYn} can not hold. Let $S\rightarrow\mathbb{P}^1$ be a minimal conic bundle surface with discriminant of degree $d = 8$, and consider the blow-down $\overline{S}\rightarrow\mathbb{F}_e$. Since $\mathbb{F}_e$ is a toric surface the cohomology of a divisor $aC_0+bF$ on $\mathbb{F}_e$ can be computed in terms of lattice objects. Indeed, when $aC_0+bF\in \Eff(\mathbb{F}_e) = \left\langle C_0,F\right\rangle$ \cite[Proposition 4.3.3]{CLS11} yields 
$$
h^0(\mathbb{F}_e,aC_0+bF) = \sum_{i = 0}^{\min\left(a,\lfloor\frac{b}{e}\rfloor\right)}(b-ie+1)
$$
when $e\geq 1$, and $h^0(\mathbb{F}_0,aC_0+bF) = (a+1)(b+1)$. In particular, for $-K_{\mathbb{F}_e} = 2C_0 +(e+2)F$ we have that $h^0(\mathbb{F}_e,-K_{\mathbb{F}_e})\geq 9$ for all $e\geq 0$ and hence, applying Remark \ref{cbc}, we get that $h^0(S,-K_{S}) = h^0(\overline{S},-K_{\overline{S}}) \geq 1$.
\end{Remark}

\begin{Remark}\label{Ottem}
If $n = 3$ and $d \leq 5$ then $-K_Z$ is birationally
pseudo-effective. Indeed, in this case $Z$ is either rational or
birational to a smooth cubic $3$-fold, see for instance
\cite[Proposition 8.9 and Theorem 9.1]{Pr18}.

It is tempting to believe that whenever $-K_{S_Z}+(n-2)F$ is pseudo-effective
$-K_Z$ is birationally pseudo-effective. Unfortunately, this is in general not
true and it is one of the reasons for which we do not have positive results on the existence of numerical Calabi-Yau conic bundles beside the one just mentioned.

Let $Z\subset\mathbb{P}^2\times\mathbb{P}^2$ be a conic bundle defined
by a smooth hypersurface of bidegree $(2,4)$. By \cite[Theorem
1.1]{Ott15} the effective cone of $Z$ is closed and generated by the
restrictions $H_{1|Z},H_{2|Z}$ of the generators of the effective cone
$\mathbb{P}^2\times\mathbb{P}^2$. Since $-K_{Z} = H_{1|Z}-H_{2|Z}$ we
see that  it is not pseudo-effective. On the other hand, by
Proposition~\ref{d12} for these conic bundles $-K_{S_Z}+F$ is
pseudo-effective.

Even when $-K_{S_Z}+F$ is effective it is not clear whether or not a section of $-K_Z$ exists, and if so how to construct it. Let $\pi_p:\mathbb{P}^{2}\dasharrow\mathbb{P}^{1}$ be the projection from a general point $p\in\mathbb{P}^{2}$, $W$ the blow-up of $\mathbb{P}^{2}$ at $p$, and $\widetilde{\pi}_p:W\rightarrow\mathbb{P}^1$ the resolution of $\pi_p$. Set $F_p:=\pi^{-1}(p)$. By blowing-up $F_p$ in $Z$ we get a conic bundle $\widetilde{\pi}:\widetilde{Z}\rightarrow W$ which, via the morphism $\widetilde{\pi}_p\circ\widetilde{\pi}:\widetilde{Z}\rightarrow \mathbb{P}^{1}$, can be viewed as a surface conic bundle $S_{\widetilde{Z}}$ over $k(t)$. Up to applying a sequence of elementary transformations to $Z$ as described in (\ref{el_3}), we may assume that the algebraic closure $\overline{S}_{\widetilde{Z}}$ of $S_{\widetilde{Z}}$ is a blow-up of $\mathbb{F}_1$. Since $d < 12$ the linear system $|-K_{\overline{S}_{\widetilde{Z}}}+F|$ has a section and hence, by Remark \ref{cbc}, $|-K_{S_{\widetilde{Z}}}+F|$ also has a section $\Gamma$ which spreads to a divisor $\widetilde{D}_{\Gamma}\subset\widetilde{Z}$. Let $D_{\Gamma}\subset Z$ be the image of $\widetilde{D}_{\Gamma}$.  

Note that, $D_{\Gamma}$ might contain the fiber $F_p$ with a certain multiplicity $m$. Hence, if $L_p\subset\mathbb{P}^{2}$ is a general line through $p$ and $S_{L_p}:=\pi^{-1}(L_p)$, we have $D_{\Gamma|S_p}\sim -K_{S_p}+(m+1)F$. To conclude that $D_{\Gamma}$ is indeed a section of $-K_Z$ one would have to prove that $m = 0$.
\end{Remark}

\section{Unirational conic bundles}\label{ur}

In the notation of (\ref{gen}), taking $d_0 = 2a_0, d_1 = a_0+a_1, d_2 = a_0+a_2, d_3 = 2a_1, d_4 = a_1+a_2, d_5 = 2a_2$, we get the conic bundle
\stepcounter{thm}
\begin{equation}\label{Cox_P}
X_{a_0,a_1,a_2} := \{\sigma_{2a_0}y_0^2+2\sigma_{a_0+a_1}y_0y_1+2\sigma_{a_0+a_2}y_0y_2+\sigma_{2a_1}y_1^2+2\sigma_{a_1+a_2}y_1y_2+\sigma_{2a_2}y_2^2 = 0\}\subset \mathcal{T}_{a_0,a_1,a_2}.
\end{equation}
We will denote by $\pi_{X_{a_0,a_1,a_2}}:X_{a_0,a_1,a_2}\rightarrow\mathbb{P}^{n-1}$ the restriction to $X_{a_0,a_1,a_2}$ of the projection $\mathcal{T}_{a_0,a_1,a_2} = \mathbb{P}(\mathcal{E}_{a_0,a_1,a_2 })\rightarrow\mathbb{P}^{n-1}$. 

The following is well-known but we were not able to find any reference.
\begin{Lemma}\label{2to1c}
Let $f:X\rightarrow\mathbb{P}^n$ be a double cover, over an algebraically closed field, ramified over an irreducible hypersurface $B_X\subset \mathbb{P}^n$ of degree four. If $n\geq 2$ then $X$ is unirational. 
\end{Lemma}
\begin{proof}
Let $L\subset\mathbb{P}^n$ be a general line and $C = f^{-1}(L)$. Denote by $\widetilde{X}$ the blow-up of $X$ along $C$ with exceptional divisor $E$, $W$ the blow-up of $\mathbb{P}^n$ along $L$, $\widetilde{f}:\widetilde{X}\rightarrow W$ the morphism induced by $f$, and $\widetilde{\pi}_L:W\rightarrow\mathbb{P}^{n-2}$ the resolution of the projection from $L$. 

Via the morphism $\widetilde{\pi}_L\circ\widetilde{f}:\widetilde{X}\rightarrow\mathbb{P}^{n-2}$ we can view $\widetilde{X}$ as a degree two del Pezzo surface $S_{\widetilde{X}}$ over $k(t_1,\dots,t_{n-2})$, see for instance \cite[Proposition 6.3.6]{Do12}.

Now, a general fiber of $E\rightarrow C$ yields a $k(t_1,\dots,t_{n-2})$-point of $S_{\widetilde{X}}$. Furthermore, since $L$ and $F$ are general we may assume that such a $k(t_1,\dots,t_{n-2})$-point is not on an exceptional curve. Therefore, by \cite[Theorem 29.4]{Ma86} $S_{\widetilde{X}}$ is unirational over $k(t_1,\dots,t_{n-2})$, and hence $X$ is unirational over the base field $k$.    
\end{proof}

Building on \cite[Example 20]{Ko17} we have the following result. 

\begin{Proposition}\label{pur}
If either $(a_0,a_1,a_2) = (c,0,0)$, or $(a_0,a_1,a_2) \in \{(a,2,0),(b,1,0)\}$ and the base field is algebraically closed then a general conic bundle of the form $X_{a_0,a_1,a_2}$ is smooth and unirational with discriminant of degree respectively $2c+3$, $2a+4$ and $2b+2$.
\end{Proposition}
\begin{proof}
Since $a_0\geq a_1\geq a_2\geq 0$ the secondary fan of $\mathcal{T}_{a_0,a_1,a_2}$ is as follows:
$$
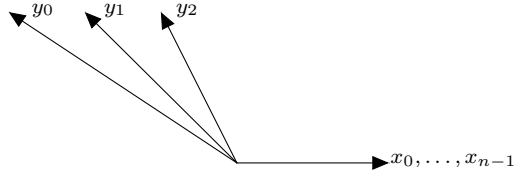
\begin{tikzpicture}[line cap=round,line join=round,>=triangle 45,x=1.0cm,y=1.0cm]
\clip(-3.5,-0.2) rectangle (3.7,2.2);
\draw [->,line width=0.4pt] (0.,0.) -- (2.,0.);
\draw [->,line width=0.4pt] (0.,0.) -- (-1.,2.);
\draw [->,line width=0.4pt] (0.,0.) -- (-2.,2.);
\draw [->,line width=0.4pt] (0.,0.) -- (-3.,2.);
\begin{scriptsize}
\draw [fill=black] (2.,0.) circle (0.0pt);
\draw[color=black] (2.85,0.029530201342281577) node {$x_0,\dots , x_{n-1}$};
\draw [fill=black] (-1.,2.) circle (0.0pt);
\draw[color=black] (-0.65,2.026845637583892) node {$y_2$};
\draw [fill=black] (-2.,2.) circle (0.0pt);
\draw[color=black] (-1.6,2.026845637583892) node {$y_1$};
\draw [fill=black] (-3.,2.) circle (0.0pt);
\draw[color=black] (-2.55,2.026845637583892) node {$y_0$};
\end{scriptsize}
\end{tikzpicture}
$$
where the nef cone $\Nef(\mathcal{T}_{a_0,a_1,a_2})$ is generated by the ray $(1,0)$, corresponding to $x_0,\dots,x_{n-1}$, and the ray $(-a_2,1)$ corresponding to $y_2$.  

The variety $X_{a_0,a_1,a_2}$ is cut out by an equation of bidegree $(0,2)$. Hence, since $a_2\geq 0$ we have that $X_{a_0,a_1,a_2}$ is either the zero locus of a section of an ample divisor if $a_2 >0$, and of a strictly nef divisor if $a_2=0$. In any case, by \cite[Theorem 6.3.12]{CLS11} $X_{a_0,a_1,a_2}\subset \mathcal{T}_{a_0,a_1,a_2}$ is cut out by a section of a globally generated divisor, and hence if such section is general by Bertini's theorem \cite[Corollary 10.9]{Ha77} $X_{a_0,a_1,a_2}$ is smooth.

Now, consider the divisor $D$ in $X_{a_0,a_1,a_2}$ defined, in terms of the presentation of $X_{a_0,a_1,a_2}$ in (\ref{Cox_P}), as follows:
\stepcounter{thm}
\begin{equation}\label{f_div}
D = \{y_0 = \sigma_{2a_1}y_1^2+2\sigma_{a_1+a_2}y_1y_2+\sigma_{2a_2}y_2^2 = 0\}\subset X_{a_0,a_1,a_2}.
\end{equation}
Hence, $\pi_{X_{a_0,a_1,a_2}|D}:D\rightarrow\mathbb{P}^{n-1}$ is a double cover of $\mathbb{P}^{n-1}$ with branch locus given by the degree $2(a_1+a_2)$ hypersurface $B_{D}:=\{\sigma_{a_1+a_2}^2-\sigma_{2a_1}\sigma_{2a_2} = 0\}\subset\mathbb{P}^{n-1}$. 

When the base field is algebraically closed we may assume that $n\geq 3$ since all surface conic bundle are rational. For the triple $(a,2,0)$ we have that $\pi_{X_{a_0,a_1,a_2}|D}:D\rightarrow\mathbb{P}^{n-1}$ is a double cover of $\mathbb{P}^{n-1}$ with branch locus given by the quartic $B_{D}:=\{\sigma_2^2-\sigma_4\sigma_0 = 0\}\subset\mathbb{P}^{n-1}$. So, Lemma \ref{2to1c} yields that $D$ is unirational. Hence, by Proposition \ref{Enr} we conclude that $X_{a_0,a_1,a_2}$ is unirational as well. 

For the triple $(b,1,0)$ the double cover $\pi_{X_{a_0,a_1,a_2}|D}:D\rightarrow\mathbb{P}^{n-1}$
is ramified over the quadric $B_{D}:=\{\sigma_1^2-\sigma_2\sigma_0 = 0\}\subset\mathbb{P}^{n-1}$. Then $D$ is rational for all $n \geq 2$, and again by Proposition \ref{Enr} we conclude that $X_{a_0,a_1,a_2}$ is unirational.  

Finally, let $k$ be an arbitrary field and consider the case $(a_0,a_1,a_2) = (c,0,0)$. Then $D$ is a divisor of bidegree $(1,2)$ in $\mathbb{P}^{n-1}\times\mathbb{P}^1$. Hence $D$ is rational and again by Proposition \ref{Enr} we conclude that $X_{a_0,a_1,a_2}$ is unirational.  
\end{proof}

\begin{Lemma}\label{ext}
Let $X_{a_0,a_1,a_2}$ be a conic bundle of the form (\ref{Cox_P}). Assume that $\sigma_{2a_i}$ is non zero for $i = 0,1,2$, $\sigma_{2a_1}$ is non constant, and either $n\geq 3$ or $n = 2$ and the base field is not algebraically closed. If $X_{a_0,a_1,a_2}$ is general then it is extremal. 
\end{Lemma}
\begin{proof}
We will prove that the inverse image in $X_{a_0,a_1,a_2}$ of any irreducible divisor in $\mathbb{P}^{n-1}$ is irreducible. Let $D\subset\mathbb{P}^{n-1}$ be an irreducible divisor and assume by contradiction that $\overline{D} = \pi_{X_{a_0,a_1,a_2}}^{-1}(D)$ is reducible. Then $\overline{D}$ must be a component of the discriminant divisor $B_{X_{a_0,a_1,a_2}}$ of $X_{a_0,a_1,a_2}$. 

Thanks to our hypotheses me may specialize $X_{a_0,a_1,a_2}$ to the conic bundle
$$
Y_{a_0,a_1,a_2} := \{\sigma_{2a_0}y_0^2+\sigma_{2a_1}y_1^2+\sigma_{2a_2}y_2^2 = 0\}\subset \mathcal{T}_{a_0,a_1,a_2}
$$
where $\sigma_{2a_i}$ is non zero for $i = 0,1,2$ and $\sigma_{2a_1}$ is non constant. The discriminant divisor $B_{X_{a_0,a_1,a_2}}$ specializes to $B_{Y_{a_0,a_1,a_2}} = \{\sigma_{2a_0}\sigma_{2a_1}\sigma_{2a_2} = 0\}\subset\mathbb{P}^{n-1}$. Therefore $D_i =\pi_{Y_{a_0,a_1,a_2}}^{-1}(\{\sigma_{2a_i}=0\})$ must be reducible for some $i$. On the other hand, the hypotheses on the $\sigma_{2a_i}$ and the genericity of $X_{a_0,a_1,a_2}$ ensure that if $n\geq 2$ then $D_i$ is irreducible for all $i$. 

If $n = 2$ the conic bundle $Y_{a_0,a_1,a_2}$ might have geometrically reducible fibers defined over the base field $k$ but when $Y_{a_0,a_1,a_2}$ is general and $k$ is not algebraically closed such geometrically reducible fibers will be irreducible over $k$.      
\end{proof}

\begin{Remark}\label{ur_hd}
Proposition \ref{pur} and Lemma \ref{ext} say, keeping in mind the flatness issue in Remark~\ref{con_fib}, that for $n\geq 3$ there are extremal unirational smooth $n$-fold conic bundles with discriminant of arbitrarily high degree, and that this fact also holds for $n = 2$ when the base field is not algebraically closed.

The construction in the proof of Proposition \ref{pur}, with $a_1 =3$ and $a_2 = 0$,
has been introduced by J. Koll\'{a}r in \cite[Example 20]{Ko17} in
order to prove that there are smooth $3$-fold conic bundles $X_{a_0,a_1,a_2}$, with
smooth discriminant of arbitrarily high degree, such that $-K_{X_{a_0,a_1,a_2}}$ is
effective. Indeed, when $a_1 = 3,a_2 = 0$ the surface $D$ in the proof of
Proposition \ref{pur} is branched over a sextic. Hence, $D$ is a $K3$
surface providing a section of $-K_{X_{a_0,a_1,a_2}}$.

More generally, in any dimension $n\geq 2$, there are smooth $n$-fold conic bundles with discriminant of arbitrarily high degree and effective anti-canonical divisor. Indeed, since
$$-K_{\mathcal{T}_{a_0,a_1,a_2}} = (n-a_0-a_1-a_2)H_1+3H_2$$
and $X_{a_0,a_1,a_2}$ is cut out by an equation of bidegree $(0,2)$ we get that
$$-K_{X_{a_0,a_1,a_2}} = (n-a_0-a_1-a_2)H_{1|X_{a_0,a_1,a_2}}+H_{2|X_{a_0,a_1,a_2}}.$$
In particular, for $a_1 = n,a_2=0$ we get $-K_{X_{a_0,a_1,a_2}} = -a_0H_{1|X_{a_0,a_1,a_2}}+H_{2|X_{a_0,a_1,a_2}}$. To conclude it is enough to note that in this case the divisor $D\subset X_{a_0,a_1,a_2}$ in (\ref{f_div}) is cut out on $X_{a_0,a_1,a_2}$ by $\{y_0=0\}$ which has bidegree $(-a_0,1)$.
\end{Remark}

\bibliographystyle{amsalpha}
\bibliography{Biblio}

\end{document}